\documentclass{amsart}
\usepackage{amssymb}

\usepackage{amsfonts}
\usepackage[english]{babel}

\usepackage[T1]{fontenc}
\usepackage{latexsym}

\usepackage{color}
\usepackage[normalem]{ulem}

\newtheorem{theorem}{Theorem}[section]
\newtheorem{lemma}[theorem]{Lemma}
\newtheorem{proposition}[theorem]{Proposition}
\newtheorem{corollary}[theorem]{Corollary}
\newtheorem{definition}[theorem]{Definition}
\newtheorem{remark}[theorem]{Remark}

  \usepackage{fancyhdr}
    \usepackage{lastpage}
\author{Maciej Wi\'sniewolski }
\address{Institute of Mathematics, University of Warsaw,  Banacha 2 \\
 02-097 Warszawa, Poland;\\ e-mail: {\tt  wisniewolski@mimuw.edu.pl }  }

\title[Inside the nature of squared Bessel process]
 {Inside the nature of squared Bessel process}

\begin{document}

\maketitle
\begin{center}
{\small
 Institute of Mathematics, University
of Warsaw \\
  Banacha 2, 02-097 Warszawa, Poland \\
 e-mail:  {\tt M.Wisniewolski@mimuw.edu.pl} }

\end{center}
\begin{abstract} A new stochastic process is introduced and considered - squared Bessel process with special stochastic time. The analogues of fundamental properties for Brownian motion are deduced for squared Bessel process. In particular an analogue of the celebrated strong Markov construction of Brownian motion independent of a given sigma field is presented and proved. This result has strong consequences. It allows for deeper understanding the nature of first hitting time of squared Bessel process. The joint distribution of two correlated squared Bessel processes is presented. For squared Bessel process it is also established a new interesting time inversion result. It is presented a general formula that ties squared Bessel process, geometric Brownian motion and its additive functional and that is generalization of the known Lamperti's relation. The new introduced process enables to find the conditional distribution of the first hitting time of a squared Bessel process with nonnegative index. Finally a completely new method of finding the density of first hitting time of squared Bessel process with nonnegative index is presented.
\end{abstract}

\noindent
\begin{quote}

\ \\


\end{quote}

\section{Introduction}
Time changing and use of so called Bessel clock were among plenty of techniques used to explain the distribution of Brownian motion functionals. In series of works
Yor, Matsumoto and co-authors delivered results for distributional dependencies betwen $A_t = \int_0^te^{2B_u +2\mu u }du$ and $e^{B_t + \mu t}$ using the unique properties of squared Bessel processes (BESQ) (see for instance \cite{DMY}, \cite{MYI}, \cite{MYII}, \cite{MYIV}, \cite{ADY}). These studies have given a reflection of the great and still growing importance BESQ processes play in distributional explanation of Brownian motion functionals. In this work we follow these studies. We introduce the new process with special stochastic time that turns out to play some surprising role in the understanding the distribution and path behaviour of squared Bessel process. To give an example - the very fundamental result coming from strong Markov property for Brownian motion is that for a stopping time $\tau$ the process $(B_{\tau + t} - B_{\tau})$ is another Brownian motion independent of $\mathcal{F}_{\tau}$ on some filtered probability space with Brownian motion $B$. The natural question one may ask is if the similar construction can be deduced for another fundamental processes for instnace for squared Bessel process. In this work we answer this question positively for a squared Bessel process with nonnegative index. The answer comes along with the mentioned new introduced stochastic process which we call squared Bessel time process. If $R$ is a squared Bessel process with nonnegative index and a $(\mathcal{F}_t)$- stopping time $\tau$, then the process
$$
	\check{R}_a = \frac{R(\tau + aR(\tau))}{R(\tau)}, \ a\geq 0
$$
turns out to be a new squared Bessel process independent of $\mathcal{F}_{\tau}$.
We present the natural example where squared Bessel time process occurs - geometric Brownian motion. It is known that the triple of geometric Brownian motion, its additive functional and squared Bessel process constitute some closed system known as Lamperti's relation (see for instance \cite{LI}, \cite{LII}, \cite{MYIV}). We discover that the new introduced process is strictly connected with this system. We present that Lamperti's relation is a particular case of even more general identity that ties the triple. This observation definitely extends the understanding of the relation between geometric Brownian motion, it's additive functional and related squared Bessel process.

\ \\
We present some interesting application of the squared Bessel time process: we establish the joint law of two correlated squared Bessel processes. This is important result as the popular in appliactions CIR process is nothing but squared Bessel process with deterministic changed time (see for instance Section 6.3 in \cite{JYC}). Another application presented is the new time inversion result for squared Bessel process. The review and examples of time inversion results for Markov processes can be found in \cite{MYII}. The new result we present in this paper is that the law of inversed squared Bessel process is the same as the law of (scaled) squared Bessel process starting from the random point. 

\ \\
Finally, an effective application of the introduced process is the description of distribution and conditional (and hence joint) distribution of first hitting times for squared Bessel process. The first hitting time is an important object in theory of stochastic processes. In particular the distribution of first hitting time is crucial in financial mathematics for problem of barrier options valution. The study of the first hitting time for a Bessel process has been undertaken recently by Hamana and Matsumoto in \cite{HM1} and \cite{HM2}. They found the solution for an appropriate differential equation tied to the generator of the Bessel process and used the fundamental fact that the Laplace transform of hitting time is such a solution (see also Kent \cite{K}). After inverting the Laplace transform, which is very hard computational task using among others theory of meromorphic functions and other complex analysis techniques, they obtained explicit expressions for the densities by means of the zeros of the Bessel functions. Some sharp estimates for the density of first hitting time of Bessel process were presented also in \cite{BMR}. Instead of considering the first hitting time for Bessel process we consider the first hitting time of its square - namely squared Bessel process. We use the squared Bessel time process to establish the conditional distribution $\mathbb{P}(\tau_y \leq t|R(T) = x)$, where $\tau_y$ is the first hitting time of $y$ by a squared Bessel process $R$ and $t\leq T$. We present completely new method of obtaining the density of $\tau_y$ by a solution of special inhomogenuous Volterra equation of the first kind. The interesting property of trajectory of squared Bessel process is the deduced:  if $t>0$ and a starting from $1$ BESQ at time $t$ satisfies $R(t) > 1$, there is almost sure a point on interval $(0,t]$ such that $R$ reaches $1$ at this point. 

\section{Bessel time process}
In the sequel $R$ denotes a squared Bessel process (BESQ), considered when needed on the canonical space of all continuous functions on $[0,\infty)$ along with its Borel sets $(\mathcal{C}([0,\infty),\mathcal{B})$. By $\mathbb{P}_{x}$ we denote the measure for coordinate process such that $R(0) = x\geq 0$ - we are using the standard notation for a Markov family indexed by a starting point $x$ (see for instance Chapter III in \cite{RY}). For simplicity we will use $\mathbb{P}$ instead of $\mathbb{P}_{1}$ (for the special case $x=1$).

\ \\
When canonical space is considered - $(\mathcal{F}_t)_{t\geq 0}$ denotes canonical filtration generated by coordinate process, augmented to satisfy usual conditions. On the space with a squared Bessel process $R$, $(\mathcal{F}^R_t)_{t\geq 0}$ denotes filtration generated by $R$ and augmented to satisfy usual conditions. On the canonical space $\mathcal{C}([0,\infty)\rightarrow \mathbb{R})$ we follow the standard notation where $\theta$ denotes the shift operator. For a squared Bessel process $R$ we assume that index $\mu$ is nonnegative, which assures that the set $\{0\}$ is polar. If $\mu > 0$ the process is transient and never reaches $0$ (for a detailed disscusion see Section 1 Chapter XI in \cite{RY}).

\begin{definition} Let $R$ be a squared Bessel process of index $\mu\geq 0$ starting from $1$. Let $a\geq 0$. We define a Bessel time process as
\begin{equation}\label{BTP}
	X(t,a) = \frac{R(t+aR(t))}{R(t)}, \ t\geq 0.
\end{equation}
\end{definition}
It is clear from the definition that $(X(t,a))_t$ is not an $(\mathcal{F}^R_t)_{t\geq 0}$-adapted process. However, it is a functional of $R$ and it turns out that it inherits some markovian features. Let us start from the following elementary lemma.
\begin{lemma} \label{shift} Fix $a\geq 0$. We have on the canonical space 
\begin{equation} \label{XR}
	X(t+s,a) = X(t,a)\circ\theta_s,
\end{equation}
for $s,t \geq 0$. Moreover for any positive Borel-measureable function $f$ on $[0,\infty)$
\begin{equation}\label{fX}
	\mathbb{E}(f(X(t+s,a))|\mathcal{F}_s) = \mathbb{E}_{R(s)}f(X(t,a)).
\end{equation}
\end{lemma}
\begin{proof} Let $R$ be the coordinate process. In our notation $R$ is under $\mathbb{P}$ the squared Bessel process such that $R(0) = 1$ and $\mathcal{F}_t$ takes the role of $\mathcal{F}^R_t$.
We have
\begin{align*}
	X(t+s,a)(\omega) &= \frac{\omega(t+s+a\omega(t+s))}{\omega(t+s)}
	= \frac{R(t+aR(t))}{R(t)}(\theta_s (\omega)) = X(t,a)\circ\theta_s(\omega).
\end{align*} 
Using Markov property of $R$ for a positive Borel-measureable function $f$ we can write
\begin{align*}
	\mathbb{E}(f(X(t+s,a)|\mathcal{F}_s) &= \mathbb{E}(f(X(t,a)\circ\theta_s)|\mathcal{F}_s)\\
	 &= \mathbb{E}\Big(f\Big(\frac{R(t+aR(t))}{R(t)}\circ\theta_s\Big)\Big|\mathcal{F}_s\Big)\\
	 &= \mathbb{E}_{R(s)}f(X(t,a)),
\end{align*}
which gives equality (\ref{fX}).
\end{proof}

\ \\
For next proposition let us introduce some auxiliary notation. Let $R^{v^2}$ be a BESQ($\mu$) starting from $v^2 (v \geq 0)$. Let us define $$X^{v^2}(t,a) = \frac{R^{v^2}(t+aR^{v^2}(t))}{R^{v^2}(t)},$$
where $a\geq 0$ is fixed. In case $v =1$ we use $R$ instead of $R^1$ and $X^1(t,a) = X(t,a)$ defined by (\ref{BTP}).
\begin{proposition} \label{PMS} For a fixed $a\geq 0, t \geq 0$, a positive Borel-measureable function $f$ on $[0,\infty)$ and any $v\geq 0$
\begin{equation}
	\mathbb{E}(f(X^{v^2}(t,a))) = \mathbb{E}f\Big(R(a)\Big).
\end{equation}
\end{proposition}
\begin{proof}
Let $t> 0$ and $r^v_t = \sqrt{R^{v^2}_t}$ be a Bessel process. The density $p_t$ of a Bessel process $r^v$ is obtained from that of BESQ by a straightforward change of variables (see Section 1 in Chapter XI \cite{RY}).
\begin{equation} \label{transdens}
\mathbb{P}(r^v_t\in dx) = p_t(v,x)dx =  \frac{1}{t}\Big(\frac{x}{v}\Big)^{\mu}xe^{-\frac{x^2+v^2}{2t}}I_{\mu}\Big(\frac{xv}{t}\Big)dx,
\end{equation}
where $I_{\mu}$ is a modified Bessel function. 
Let us look at $X^{v^2}(t,a)$ as a functional of the Bessel process $r^v_t$. We have
\begin{align*}
 &\mathbb{E}f\Big(X^{v^2}(t,a)\Big) = \int_0^{\infty}\int_0^{\infty}f\Big(\frac{y^2}{x^2}\Big)p_t(v,x)p_{ax^2}(x,y)dxdy\\
 &= \int_0^{\infty}\int_0^{\infty}f\Big(\frac{y^2}{x^2}\Big)\frac{1}{t}\Big(\frac{x}{v}\Big)^{\mu}xe^{-\frac{x^2+v^2}{2t}}I_{\mu}\Big(\frac{xv}{t}\Big) \frac{1}{ax^2}\Big(\frac{y}{x}\Big)^{\mu}ye^{-\frac{y^2+x^2}{2ax^2}}I_{\mu}\Big(\frac{y}{ax}\Big)dxdy\\
 &= \int_0^{\infty}\int_0^{\infty}f(\tilde{y}^2)\frac{1}{t}\Big(\frac{1}{v}\Big)^{\mu}x^{\mu+1}e^{-\frac{x^2+v^2}{2t}}I_{\mu}\Big(\frac{xv}{t}\Big) \frac{1}{a}(\tilde{y})^{\mu+1}e^{-\frac{\tilde{y}^2+1}{2a}}I_{\mu}\Big(\frac{\tilde{y}}{a}\Big)d\tilde{y}dx,
\end{align*}
where in the last equality we use Fubbini's theorem and substitution $\tilde{y} = y/x$. From the last expression we obtain
\begin{align*}
\mathbb{E}f(X^{v^2}(t,a)) &= \int_0^{\infty}\frac{1}{t}\Big(\frac{1}{v}\Big)^{\mu}x^{\mu+1}e^{-\frac{x^2+v^2}{2t}}I_{\mu}\Big(\frac{xv}{t}\Big)dx\int_0^{\infty}\frac{f(\tilde{y}^2)}{a}(\tilde{y})^{\mu+1}e^{-\frac{\tilde{y}^2+1}{2a}}I_{\mu}\Big(\frac{\tilde{y}}{a}\Big)d\tilde{y}\\
	&= \mathbb{E}f\Big(R(a)\Big).
\end{align*}
\end{proof}

\ \\
For next considerations let us
denote $$H^f(v^2,t,a) = \mathbb{E}(f(X^{v^2}(t,a)))$$
and observe that from Proposition \ref{PMS} $H^f(v^2,t,a)$ does not depend on $v^2$.
\begin{corollary}\label{cor1} Let $a\geq 0$. For fixed $t\geq 0$, on the canonical space
\begin{equation}
	X(t,a)  \stackrel{(law)}{=} R(a).
\end{equation} 
Moreover $X(t,a)$ is independent of $\mathcal{F}_t$ and the
two above facts hold for any $(\mathcal{F}_t)_{t\geq 0}$- stopping time $\tau$.
\end{corollary}
\begin{proof} By previous proposition distribution of $X(t,a)$, for fixed $t\geq 0$ and $a\geq 0$
does not depend on $t$ nor on the starting point of $R^{v^2}$.
As a consequence for every $s,t\geq 0$ random variable $X(t+s,a)$ does not depend on $\mathcal{F}_s$. Indeed, for any positive Borel-measureable function $f$ on $[0,\infty)$ from
(\ref{fX}) and observation about $H$ made before corollary we obtain
\begin{align*}
	\mathbb{E}(f(X(t+s,a))|\mathcal{F}_s) &= \mathbb{E}_{R(s)}f(X(t,a)) = H^f(R(s),t,a)\\
	 &= \mathbb{E}f\Big(R(a)\Big).
\end{align*}
To conclude that $X(t,a)$ is independent of $\mathcal{F}_t$ it is enough in last equality to put $0$ instead of $t$ and $t$ instead of $s$.

Let us now observe that due to strong markovianity of $R$,
the above facts hold for any $(\mathcal{F}_t)_{t\geq 0}$- stopping time $\tau$. 
Indeed
\begin{align*}	
\mathbb{E}\Big(f(X(\tau,a))\Big|\mathcal{F}_{\tau}\Big)
	&= \mathbb{E}\Big(f((R(aR(0))/R(0))\circ\theta_{\tau})\Big|\mathcal{F}_{\tau}\Big)\\
	&= H^f(R(\tau),0,a) = \mathbb{E}\Big(f(R(a))\Big),
\end{align*}
where in the second equality we used the scaling property of BESQ, namely: for any $x>0$ the process $(\frac{1}{x}R^x(xt), t\geq0)$ has the same distribution as $(R^1(t), t\geq 0)$ (see Chapter XI Prop. 1.6 in \cite{RY}).
In result $X_{\tau}$ is independent of $\mathcal{F}_{\tau}$.
\end{proof}

\ \\
Let us define by $\tau_y$ the first moment of reaching $y>0$ by $R$, i.e. 
\begin{align*}
	\tau_y = \inf\{u>0: R(u) = y\}.
\end{align*}
\begin{proposition} Let $y>0$. For every $\epsilon > 0$
\begin{equation}
	R(\tau_y + \epsilon y) \stackrel{(law)}{=} y R(\epsilon).
\end{equation} 
Moreover for $t>0$ and any positive Borel-measureable $f$ we have 
\begin{equation}
	\mathbb{E}\Big(f(R(\tau_y + \epsilon y))\Big|\tau_y > t\Big) = \mathbb{E}f(yR(\epsilon)).
\end{equation}
\end{proposition}
\begin{proof} For $\epsilon>0$ consider $X(t,\epsilon)$. Let $f$ be Borel and positive. From the strong Markov property of $R$ we have
\begin{align*}
	\mathbb{E}f\Big(\frac{R(\tau_y + \epsilon y)}{y}\Big) &= \mathbb{E}f(X(\tau_y,\epsilon))\\
	 &= \mathbb{E}f(X(0,\epsilon)\circ\theta_{\tau_y})\\
	&= H^f(y,0,\epsilon)\\
	 &= \mathbb{E}f(R(\epsilon)).
\end{align*}
For the second part of the proposition observe that from Corollary \ref{cor1} $\frac{R(\tau_y + \epsilon y)}{y}$ has a BESQ distribution at time $\epsilon$ independent of $\tau_y$.
In result 
\begin{align*}
	\mathbb{E}1_{\{\tau_y > t\}}f\Big(R(\tau_y + \epsilon y)\Big) &= 
	 \mathbb{E}\Big(1_{\{\tau_y > t\}}f\Big(\frac{R(\tau_y + \epsilon y)}{y}y\Big)\Big)\\
	&= \mathbb{P}(\tau_y>t)\mathbb{E}f(yR(\epsilon)),
\end{align*}
where in the second equality we used again Corollary \ref{cor1}.
\end{proof}

The next theorem tells that having obtained $y$, in a short time, process $R$ escapes in some sense arbitrarily far away from the point $y$.
\begin{theorem} Let $\delta > 0$. Then
\begin{equation}
	\lim_{\epsilon\rightarrow 0}\mathbb{P}\Big(|R(\tau_y + \epsilon y)- y|\leq \delta\Big) = 0.
\end{equation}
\end{theorem}
\begin{proof} From the previous proposition 
\begin{align*}
	\mathbb{P}\Big(|R(\tau_y + \epsilon y)- y|\leq \delta\Big) &= \mathbb{P}\Big( \frac{(y-\delta)^+}{y} \leq R(\epsilon) \leq \frac{y+\delta}{y}\Big)\\
	&=\int_{\sqrt{\frac{(y-\delta)^+}{y}}}^{\sqrt{\frac{y+\delta}{y}}} \frac{1}{\epsilon}x^{\mu+1}e^{-\frac{x^2+1}{2\epsilon}}I_{\mu}\Big(\frac{x}{\epsilon}\Big)dx.
\end{align*}
Let us denote
$$
	f_x(t) = \frac{1}{t}e^{-\frac{x^2+1}{2t}}I_{\mu}\Big(\frac{x}{t}\Big).
$$
Observe that for $x\neq 1$ we have $\lim_{t\rightarrow 0} f_x(t) = \lim_{t\rightarrow \infty} f_x(t) = 0$.
Thus the assertion follows now from continuity of $f_x$ and the Lebesgue theorem.
\end{proof}
Now let us fix $t\geq 0$ and look at the process $(X(t,a))_{a\geq 0}$ with respect to filtration $(\mathcal{F}_a^X)_{a\geq 0}$ where $\mathcal{F}^X_a = \sigma(X(t,u):u\leq a)$. It is clear that $X(t,0) = 1$. The surprising and less obvious is the next:
\begin{theorem} \label{Besa} For any fixed $t\geq 0$ the process $(X(t,a))_{a\geq 0}$ is starting from $1$ and independent of $\mathcal{F}_t$ BESQ with index $\mu \geq 0$.
\end{theorem}
\begin{proof} We will show that positive process $r$ given by $(r_a^2)_{a\geq 0} = (X(t,a))_{a\geq 0}$ is a Markov process with transition density given by (\ref{transdens}) and such that $X(t,0) = 1$. Let $p$ be a transition function of $\sqrt{R}$. Let us again look at the $X(t,a)$ as a functional of Bessel process $\sqrt{R}$ (see the proof of Proposition \ref{PMS}). Choose $0\leq a_1 < a_2 < .. < a_n$ and
$f_1, f_2, ..,f_n$, where each $f_i$ is a positive Borel function. Observe that
\begin{align*}
	&\mathbb{E}\Pi_{i=1}^nf_i(r_{a_i}^2) = \mathbb{E}\Pi_{i=1}^nf_i(X(t,a_i))=\\
	&=\int_0^{\infty}..\int_0^{\infty}p_t(1,x)f_1\Big(\frac{x_1^2}{x^2}\Big)p_{a_1x^2}(x,x_1)..f_n\Big(\frac{x_n^2}{x^2}\Big)p_{(a_n-a_{n-1})x^2}(x_{n-1},x_n)dx_n..dx_1dx\\
	&= \int_0^{\infty}..\int_0^{\infty}p_t(1,x)f_1(\hat{x}_1^2)p_{a_1}(1,\hat{x}_1)f_2(\hat{x}_2^2)p_{a_2-a_1}(\hat{x}_1,\hat{x}_2)..\\
	&\times f_n(\hat{x}_n^2)p_{a_n-a_{n-1}}(\hat{x}_{n-1},\hat{x}_n)d\hat{x}_n..d\hat{x}_1dx,
\end{align*}
where we substituted $\hat{x}_i = \frac{x_i}{x}$ and used the fact that for any $a >0$ $$p_{ax^2}(y,z) = \frac{1}{x^2}p_{a}(y/x,z/x).$$
Finally, $X(t,a)$ is independent of $\mathcal{F}_t$ from Corollary \ref{cor1}.
\end{proof}

\begin{remark} \label{stop} The last theorem shows that having a BESQ($\mu$), let's say $R$,  for a fixed $t\geq 0$ we can construct another BESQ $\check{R}_a = \frac{R(t+aR(t))}{R(t)}$ which 
 is again BESQ with index $\mu$, but independent of $\mathcal{F}_t$. This is what we announced in introduction. When one looks closer one can see that this is an analogue construction to $(B_{\tau + t} - B_{\tau})$, $B$ - standard Brownian motion, but for BESQ process. From strong Markov property for BESQ follows that $t$ in the definition of $\check{R}_a$ can be replaced by a stopping time $\tau$. Let us observe that for a process $R$ the usual procedure of adding a cemetary point (to define $R(\tau)$ on the whole space $\Omega$) and puting $R(\tau) = \Delta$ on $\{\tau = \infty\}$ may be forgotten as the death point of $R$ is infinite a.s. ($\zeta = \infty$ - see also remark after Theorem 3.1 in Chapter III and Exercise 2.10 in Chapter IX of \cite{RY}).
\end{remark}

\begin{theorem}\label{BesaM} Let $\tau$ be an $(\mathcal{F}_t)$- stopping time. Then the process $(X(\tau,a))_{a\geq 0}$ is  BESQ with index $\mu \geq 0$ starting from $1$ and independent of $\mathcal{F}_{\tau}$.
\end{theorem}
\begin{proof} From Theorem \ref{Besa} the assertion holds for $\tau$ taking values in a countable set. To get the general case we introduce $\tau_n = \frac{[2^n\tau] +1}{2^n}, n\geq 0$ which is a sequence of stopping times with values in countable set and decreasing to $\tau$. To finish the proof we proceed as in proof of strong Markov property and applicate the monotone class Theorem (see Theorem 3.1 Chapter III in \cite{RY}).
\end{proof}

Having constructed a new squared Bessel process $(\check{R})_{a\geq 0} = (X(t,a))_{a\geq 0}$ independent of a given filtration $(\mathcal{F}_t)$ we may wonder if some analogues of reflection principle for Brownian motion occur for BESQ. As we can see in the next theorem the answer is positive. Moreover, we will see in the Section 4 that this analogue of principle reflection has very strong consequence, as it allows in completely new way establish the density of the first hitting time of squared Bessel process.

\begin{theorem} \label{refprinc} Let $R,\hat{R}$ be two independent copies of BESQ with index $\mu\geq 0$ and starting from $1$. 
Fix $t>0$.
Then for every $0< y \leq b$ 
\begin{equation}
	\mathbb{P}(R(t) \geq b) = \mathbb{P}\Big(\hat{R}\Big(\frac{t-\tau_y}{y}\Big)\geq \frac{b}{y}, \tau_y \leq t\Big).
\end{equation}
\end{theorem}
\begin{proof} For $0< y \leq b$, we have
\begin{align*}
	\mathbb{P}(R(t) \geq b) &= \mathbb{P}(R(t) \geq b, \tau_y\leq t)\\
	&= \mathbb{P}\Big(\frac{R(\tau_y + a^yy)}{y}\geq \frac{b}{y}, \tau_y \leq t\Big),
\end{align*}
where $a^y = (t-\tau_y)/y\in\mathcal{F}_{\tau_y}$ is well defined and nonnegative on $\{\tau_y \leq t\}$. For every $a\geq 0$ from the last remark $\check{R}_a = \frac{R(\tau_y +ay)}{y}$ is independent of $\mathcal{F}_{\tau_y}$. Thus
\begin{align*}
\mathbb{P}\Big(\frac{R(\tau_y+a^yy)}{y}\geq \frac{b}{y}, \tau_y \leq t\Big) &= \mathbb{P}\Big(\check{R}_{a^y}\geq \frac{b}{y}, \tau_y \leq t\Big)\\
	&= \mathbb{P}\Big(\hat{R}\Big(\frac{t-\tau_y}{y}\Big)\geq \frac{b}{y}, \tau_y \leq t\Big).
\end{align*}
\end{proof}

\ \\
In fact we can state more about the path of BESQ after hitting the point $y>0$:
\begin{proposition}\label{StopTime}	For any $\alpha \geq 1$
\begin{equation}
	R(\alpha\tau_y) = y\check{R}_{\tau_y\frac{\alpha-1}{y}},
\end{equation}
where $\check{R}$ is BESQ($\mu$) independent of $\mathcal{F}_{\tau_y}$. 
\end{proposition}
\begin{proof}From Corollary \ref{cor1} and Theorem \ref{BesaM} $(\check{R}_a,a\geq 0) = (X(\tau_y,a),a\geq 0)$ is the squared Bessel process with index $\mu$ and independent of $\mathcal{F}_{\tau_y}$. From Definition \ref{BTP} we have
$$
	\check{R}_a  = X(\tau_y,a) = \frac{R(\tau_y+ay)}{y}.
$$
As $\check{R}_a$ is independent of $\mathcal{F}_{\tau_y}$ we have for $a = \tau_y\frac{\alpha-1}{y}$
$$
	\check{R}_{\tau_y\frac{\alpha-1}{y}} =  \frac{R(\alpha \tau_y)}{y}.
$$
\end{proof}

\begin{proposition} \label{plus} For fixed $t\geq 0$ and any $s\geq 0$
\begin{equation}
	R(t + s) =\check{R}^{R(t)}_s,
\end{equation}
where $\check{R}^x$ is BESQ($\mu$) independent of $R$ starting from $x\geq 0$.
\end{proposition}
\begin{proof} Observe that $(X(t,a),a\geq 0)$ is independent of $\mathcal{F}_t$ by Theorem \ref{Besa}. For any $s\geq 0$ and $a = \frac{s}{R(t)}$ we have
$$
	\frac{R(t+s)}{R(t)} =	X\Big(t,\frac{s}{R(t)}\Big) = \check{R}_{\frac{s}{R(t)}},
$$
where $\check{R}$ is independent of $\mathcal{F}_t$ squared Bessel process. From the last equality
$$
	R(t+s) = R(t)\check{R}_{\frac{s}{R(t)}}.
$$
Assertion follows now from scaling property of BESQ and independency of process $\check{R}$ and random variable $R(t)$.
\end{proof}

\begin{proposition} For fixed $t\geq 0$ and any bounded measurable functional
$F_t$ on $\mathcal{C}([0, t];\mathbb{R}_+)$ 
\begin{equation}
	\mathbb{E}F_t(R) R(t+aR(t)) = (1+\delta a)\mathbb{E}F_t(R)R(t).
\end{equation}
In result
\begin{equation}
	\mathbb{E}(R(t+aR(t))|\mathcal{F}_t) = (1+\delta a)R(t)
\end{equation}
and
\begin{equation}
	\mathbb{E}R(t+aR(t)) = (1+\delta a)(1+\delta t).
\end{equation}
\end{proposition}
\begin{proof} From Corollary \ref{cor1} $X(t,a)$ is independent of $\mathcal{F}_t$. Thus
\begin{align*}
	\mathbb{E}(F_t(R)R(t+aR(t))) &= \mathbb{E}(F_t(R)R(t)X(t,a))\\
	 &= \mathbb{E}(F_t(R)R(t))\mathbb{E}X(t,a)\\
	 &= (1+\delta a)\mathbb{E}(F_t(R)R(t)).
\end{align*}
\end{proof}

The application of the squared Bessel time process leads us to the construction of two correlated squared Bessel processes with explicitly known joint distribution.
In fact it is the answer to the question how the joint distribution of the two correlated Bessel processes looks like. The question arose in financial mathematics and is strictly connected with the stochastic volatility models, where an asset price process and its variance are two correlated stochastic processes (see \cite{JYC}). For instance popular among practitioners CIR process is a squared Bessel process with deterministic changed time. However, as far as we know, the joint distribution of two correlated CIR processes is not known. Although the correlation structure in stochastic volatility models is usually imposed on the dynamics of the two processes, we belive our construction gives some new light on this subject. 
\begin{theorem} 
For a fixed $t\geq 0$ define
\begin{align*}
(\check{R}_a, a\geq 0) &= (X(t,a),a\geq 0),\\ 
(U_a,a \geq 0) &= (R(t+a),a \geq 0).
\end{align*}
Then $(\check{R}_a, U_a)$ is the pair of the two corellated squared Bessel processes with index $\mu$ and covariance given by
\begin{equation}
	Cov(\check{R}_a, U_a) = \mathbb{E}\Big[R(t)H\Big(a,\frac{a}{R(t)},\delta\Big)\Big] - (1+\delta a)(1 +\delta t + \delta a),
\end{equation}
where
\begin{equation}
	H(s,t,\delta) = 1 + (t\wedge s)(\delta +4) + 2\delta (t\wedge s)^2 + \delta (t\vee s)(1+ \delta (t\wedge s)).
\end{equation}
Moreover for any positive, Borel functions $f,g$ and $R^{(1)}, R^{(2)}, R^{(3)}$ three independent, starting from $1$ BESQ processes with index $\mu$ we have
\begin{align*}
	\mathbb{E}f(\check{R}_a)g(U_a) &= \mathbb{E}1_{\{R^{(1)}_t\geq 1\}}f\Big(R^{(2)}_{a/R^{(1)}_t}R^{(3)}_{a(R^{(1)}_t-1)/(R^{(1)}_tR^{(2)}_{a/R^{(1)}_t})}\Big)g\Big(R^{(1)}_tR^{(2)}_{a/R^{(1)}_t}\Big)\\
	&+ \mathbb{E}1_{\{ R^{(1)}_t < 1\}}f(R^{(2)}_a)g\Big(R^{(1)}_tR^{(2)}_aR^{(3)}_{(1-R^{(1)}_t)a/(R^{(1)}_tR^{(2)}_a)}\Big).
\end{align*}
\end{theorem}
\begin{proof} From Theorem \ref{Besa} $\check{R}_a$ is a squared Bessel process with index $\mu$ independent of $\mathcal{F}_t$. Thus $\check{R}_a$ and $R(t)$ are independent, and in result from scaling property of BESQ, for fixed $t\geq 0$ the process $(U_a = R(t)\check{R}_{\frac{a}{R(t)}},a\geq 0)$ is BESQ with index $\mu$ starting from the random point $R(t)$. From the definition of $\check{R}_a$ we have
$$U_a = R(t)\check{R}_{\frac{a}{R(t)}} = R(t)\frac{R(t+a)}{R(t)} = R(t+a).$$
 Observe that from Markov property for a BESQ $R$ starting from $1$ and $t,s\geq 0$ we obtain
\begin{align*}
	\mathbb{E}R(t)R(s) &= \mathbb{E}\Big(R(t\wedge s)\mathbb{E}\Big[R(t\vee s)|\mathcal{F}^R_{t\wedge s}\Big]\Big)\\
	 &= \mathbb{E}\Big(R(t\wedge s)\mathbb{E}_{R(t\wedge s)}R(t\vee s - t\wedge s)\Big)\\
	&= \mathbb{E}R^2(t\wedge s) + \delta (t\vee s - t\wedge s)(1+\delta (t\vee s))\\
	&= 1 + (\delta + 2)(2 +\delta (t\vee s))(t\vee s) + \delta (t\vee s - t\wedge s)(1+\delta (t\vee s))\\
	&= 1 + (t\wedge s)(\delta +4) + 2\delta (t\wedge s)^2 + \delta (t\vee s)(1+ \delta (t\wedge s))\\
	 &= H(s,t,\delta),
\end{align*}
where we used the fact that for 
$$R(t) = 1 +2\int_0^t\sqrt{R(u)}dB_u +\delta t,$$
we have
\begin{align*}
	\mathbb{E}R^2(s) = 1 + 2(\delta+2)\int_0^s\mathbb{E}R(u) du
\end{align*}
and the definition of function $H$.
Observe that
\begin{align*}
	\mathbb{E}\check{R}_a U_a &= \mathbb{E}\Big(R(t)\mathbb{E}\Big(\check{R}_a\check{R}_{\frac{a}{R(t)}}\Big|R(t)\Big)\Big)\\
	 &= \mathbb{E}\Big(R(t)H\Big(a,\frac{a}{R(t)},\delta\Big)\Big).
\end{align*}
In result from the last equality we obtain
\begin{align*}
	Cov(\check{R}_a, U_a) &= \mathbb{E}\check{R}_a U_a - \mathbb{E}\check{R}_a\mathbb{E}U_a\\
	 &=  \mathbb{E}\Big(R(t)H\Big(a,\frac{a}{R(t)},\delta\Big)\Big) - (1+\delta a)(1 +\delta t + \delta a).
\end{align*}
Now let $f,g$ be two positive, Borel functions. We have
\begin{align*}
	\mathbb{E}&f(\check{R}_a)g(U_a) = \mathbb{E}\Big[f\Big(\frac{R(t+aR(t))}{R(t)}\Big)g(R(t+a))\Big]\\
	&= \mathbb{E}\Big[f\Big(\frac{R(t+aR(t))}{R(t)}\Big)g(R(t+a))1_{\{R(t)\geq 1\}}\Big]\\
	 &+ \mathbb{E}\Big[f\Big(\frac{R(t+aR(t))}{R(t)}\Big)g(R(t+a))1_{\{R(t)< 1\}}\Big]\\
	&= I + II.
\end{align*}
We have further
\begin{align*}
	I &= \int_0^\infty ..\int_0^\infty 1_{\{x^2\geq 1\}}f\Big(\frac{y^2}{x^2}\Big)g(z^2)p_t(1,x)p_{a}(x,z)p_{a(x^2-1)}(z,y)dxdydz\\
	&= \int_0^\infty ..\int_0^\infty 1_{\{x^2\geq 1\}}f\Big((\tilde{z}\tilde{y})^2\Big)g\Big((\tilde{z}x)^2\Big)p_t(1,x)p_{\frac{a}{x^2}}(1,\tilde{z})p_{\frac{a(x^2-1)}{(x^2\tilde{z}^2)}}(1,\tilde{y})dxd\tilde{z}d\tilde{y}.
\end{align*}
Thus
\begin{align*}
	I = \mathbb{E}\Big[1_{\{R^1_t\geq 1\}}f\Big(R^2_{a/R^1_t}R^3_{a(R^1_t-1)/(R^1_tR^2_{a/R^1_t})}\Big)g\Big(R^1_tR^2_{a/R^1_t}\Big)\Big].
\end{align*}
In the same way we obtain
\begin{align*}
	II &= \mathbb{E}\Big[f\Big(\frac{R(t+aR(t))}{R(t)}\Big)g(R(t+a))1_{\{R(t)< 1\}}\Big]\\
	&= \int_0^\infty ..\int_0^\infty 1_{\{x^2< 1\}}f\Big(\frac{y^2}{x^2}\Big)g(z^2)p_t(1,x)p_{ax^2}(x,y)p_{a(1-x^2)}(y,z)dxdydz\\
	&= \int_0^\infty ..\int_0^\infty 1_{\{x^2< 1\}}f\Big((\tilde{y})^2\Big)g\Big((\tilde{y}\tilde{z}x)^2\Big)p_t(1,x)p_a(1,\tilde{y})p_{\frac{a(1-x^2)}{x^2\tilde{y}^2}}(1,\tilde{z})dxd\tilde{z}d\tilde{y}\\
	&= \mathbb{E}\Big[1_{\{R^1_t< 1\}}f(R^2_a)g\Big(R^1_tR^2_aR^3_{a(1-R^1_t)/(R^1_tR^2_a)}\Big)\Big].
\end{align*}
\end{proof}
The next important result is another version of time inversion for squared Bessel process (for a short review of time reversal/inversion theory see for instance Chapter VII Par. 4 in \cite{RY} or \cite{MYII}). In fact it is an answer to the following question: can we obtain any equivalence between the BESQ process and the BESQ process of the same index but with inverted time? The answer is positive. We can achieve it by making the starting point of BESQ random.

\begin{theorem} Let $R^x$ be BESQ with index $\mu\geq 0$ and starting from $x> 0$. Let $\hat{R}(t) =t^2R\Big(\frac{1}{t}\Big)$. Then for any fixed $t\geq 0$ the following equivalence in law holds
\begin{equation}
 \Big(t^2R\Big(\frac{1}{t} + a\Big), a\geq 0\Big) \stackrel{(law)}{=} \Big(S^{\hat{R}_t}_{t^2a}, a\geq 0\Big),
\end{equation}
where $S$ is independent of $\hat{R}$ squared Bessel process with index $\mu$.
\end{theorem}
\begin{proof} By Theorem \ref{Besa} $(X(t,a), a \geq 0)$ is BESQ($\mu$) starting from $1$ and independent of $\mathcal{F}_t$. Define $\hat{R}_t = t^2R\Big(\frac{1}{t}\Big)$. $\hat{R}$ is well defined markovian, transient process starting from $0$. Indeed, there exists a constant $M>0$ such that 
\begin{align*}
	R(t) \leq (1+B^{(1)}_t)^2 + \sum_{i=2}^M (B^{(i)}_t)^2 \leq 2 \Big(1+ \sum_{i=1}^M (B^{(i)}_t)^2\Big),
\end{align*}
where $B^{(i)}, i = 1,..,M$ are independent Brownian motions. Thus
\begin{align*}
t^2R\Big(\frac{1}{t}\Big)&\leq 2t^2 \Big(1+ \sum_{i=1}^M (B^{(i)}_{1/t})^2\Big) = 2t^2 + 2\sum_{i=1}^M \frac{(B^{(i)}_{1/t})^2}{1/t^2} 
\end{align*}
and obviously the last expresion tends to $0$ for $t\downarrow 0$.
For detailed study and the role of $\hat{R}$ in the time-inversion operations see \cite{MYII}. In particular $\hat{R}$ is a square of an upward Bessel process $\hat{r}$ with index $\mu\geq 0$. The processs $\hat{r}$ is a diffusion (see again \cite{MYII}) on $[0,\infty)$ with generator
$$
	G_{\mu} = \frac{1}{2}\frac{d}{dx^2} + \Big(\frac{2\mu+1}{2x}+ \delta \frac{I_{\mu +1}(\delta x)}{I_{\mu}(\delta x)}\Big)\frac{d}{dx}.
$$
For $t>0$ define $(\check{R}_a, a \geq 0) = (X(1/t,a), a \geq 0)$. It is a BESQ($\mu$) independent of $\mathcal{F}_{\frac{1}{t}}$. Now, in $\check{R}_a$, let us put $at^2$ in place of $a$ to obtain
\begin{align*}
	\check{R}_{at^2} = X(1/t,at^2) &= \frac{R\Big(\frac{1}{t} + at^2R\Big(\frac{1}{t}\Big)\Big)}{R\Big(\frac{1}{t}\Big)} = t^2\frac{R\Big(\frac{1}{t} + a\hat{R}_t\Big)}{\hat{R}_t}.
\end{align*}
In result 
\begin{equation} \label{TI}
 \frac{R\Big(\frac{1}{t} + a\hat{R}_t\Big)}{\hat{R}_t} = \frac{1}{t^2}\check{R}_{at^2}.
\end{equation}
Now observe that from Theorem \ref{Besa}, LHS of (\ref{TI}) is independent of $\mathcal{F}_{\frac{1}{t}}$, thus of $\hat{R}_t$. In result, if we replace $a$ with $\frac{a}{\hat{R}_t}$ in (\ref{TI}) we obtain
\begin{align*}
	\frac{R\Big(\frac{1}{t} + a\Big)}{\hat{R}_t} = \frac{1}{t^2}\check{R}_{\frac{at^2}{\hat{R}_t}}.
\end{align*}
Finally from scaling property of BESQ we have
\begin{equation} \label{TII}
	t^2R\Big(\frac{1}{t} + a\Big) = \check{R}^{\hat{R}_t}_{t^2a},
\end{equation}
and the process $\check{R}$ is BESQ($\mu$) independent of $\hat{R}_t$ . 
\end{proof}
\begin{corollary} (Time inversion) For any $t\geq 0$ we have equality in law 
\begin{align*}
t^2R^1\Big(\frac{2}{t}\Big) \stackrel{(law)}{=} \tilde{R}^{\hat{R}(t)}(t)
\end{align*}
where $\tilde{R}$ and $\hat{R}$ are independent. 
\end{corollary}
\begin{proof}
It is enough to put $a = \frac{1}{t}$ to obtain from the previous theorem
\begin{equation*}
 	t^2R^1\Big(\frac{2}{t}\Big) \stackrel{(law)}{=} \tilde{R}^{\hat{R}(t)}(t).
\end{equation*}
To finish the proof observe that both sides of the last equality tend to $0$ with $t\rightarrow 0$.
\end{proof}

\section{Bessel time process and geometric Brownian motion}

Although in the previous section we have already seen the surprising applications of the Bessel time process $(X(t,a),a\geq 0)$, one may wonder where the construction of squared Bessel time process can be observed. The answer is: the structure of Bessel time process is observed in decomposition of a geometric Brownian motion.
\begin{proposition} \label{GBM} Fix $h\geq 0$, $\mu \geq 0$. Let $\hat{B}_s = B_{h+s}-B_h, s\geq 0$, where $B$ is a standard Brownian motion. Then
\begin{equation}\label{XA}
	X(A_h, \hat{A}_s) = e^{2\hat{B}_s + 2\mu s},
\end{equation}
where $A_h = \int_0^he^{2B_u + 2\mu u}du$, $\hat{A}_s = \int_0^se^{2\hat{B}_u + 2\mu u}du$.
\end{proposition}
\begin{proof} Let $R$ be a BESQ($\mu$) such that $R(A_h) = e^{2B_h + 2\mu h}$. The last identity is known in literature as Lamperti's relation (for the detailed study of it see for instance \cite{Y} or \cite{MYIV}). From the definition of $X$ we have
\begin{align*}
	X(A_h, \hat{A}_s) = \frac{R(A_h+\hat{A}_s R(A_h))}{R(A_h)}.
\end{align*}
From the known decomposition $A_{h+s} = A_h + e^{2B_h + 2\mu h}\hat{A}_s$ (see \cite{DMY}, \cite{MYI})and Lamperti's relation
\begin{align*}
	\frac{R(A_h+\hat{A}_s R(A_h))}{R(A_h)} = \frac{R(A_{h+s})}{R(A_h)} = e^{2(B_{h+s}-B_h) + 2\mu s} = e^{2\hat{B}_s + 2\mu s}.
\end{align*}
\end{proof}

\ \\
Note that the process $(e^{2\hat{B}_s + 2\mu s},s\geq 0)$ is independent of $\mathcal{F}^B_h$. In the same time processes $R,B$ are adapted to the filtration $(\mathcal{F}^R_t)$.

\ \\
Let us introduce some useful notation. For a standard Brownian motion $B$ and $h\geq 0$ let
$$
\hat{B}_s(h) = B_{s+h} - B_h.
$$
As $\kappa_h = \inf\{t: \int_0^te^{2B_u +2\mu u}du \geq h\}$ is a stopping time, $(\hat{B}(\kappa_h))$ is another Brownian motion independent of $\mathcal{F}_{\kappa_h}$. Lamperti's relation implies that for any Brownian motion $B$ there exists associated squared Bessel process with index $\mu$ such that
$$
	R(A_t) = e^{2B_t + 2\mu t}, \ A_t = \int_0^te^{2B_u + 2\mu u}du.
$$
The triple $R(t),A_t,e^{2B_u + 2\mu u}$ constitutes then the closed system. 
In particular there exists a BESQ($\mu$) $\hat{R}^h$ such that
$$
	\hat{R}^h(\hat{A}_t(h)) = e^{2\hat{B}_t(h) + 2\mu t}, \ \hat{A}_t(h) = \int_0^te^{2\hat{B}_u(h) + 2\mu u}du.
$$
Recall that for $h\geq 0$
$$(\check{R}^h_a, a\geq 0) = (X(h,a), a\geq 0) = \Big(\frac{R(h + aR(h))}{R(h)},a \geq 0\Big).
$$
Next theorem states that for the Brownian motion $\hat{B}(h)$ the squared Bessel process $\check{R}^h$ is exactly the BESQ $\hat{R}^h$ associated to $\hat{B}(h)$ from Lamperti's relation (the closed system).
\begin{theorem} For $h\geq 0$ the process $\check{R}^h$ is in fact the $\hat{R}^h$ associated to $\hat{B}(h)$. In other words $\check{R}^h\equiv\hat{R}^h$.
\end{theorem}
\begin{proof} 
By (\ref{XA}) we have
$$
	e^{2\hat{B}_s(\kappa_h) + 2\mu s} = \frac{R(h+\hat{A}_s(\kappa_h) R(h))}{R(h)}.
$$
Now observe that from the definition of $\check{R}^h$
$$
	\frac{R(h+\hat{A}_s(\kappa_h) R(h))}{R(h)} = \check{R}^h_{\hat{A}_s(\kappa_h)}.
$$
On the other hand if $\hat{R}^h$ is an associated to $\hat{B}(h)$ BESQ($\mu$) from Lamperti's relation then
$$
	\hat{R}^h_{\hat{A}_s(\tau_h)} = e^{2\hat{B}_s(\tau_h) + 2\mu s}.
$$
The proof is complete.
\end{proof}

The surprising conclusion coming from the Bessel time process considerations is that Lamperti's relation is a special case of a deeper fact:
\begin{theorem} For any $s\geq 0, t\geq 0$ we have
\begin{equation}
	R(A_t + s) = \check{R}^{e^{2B_t + 2\mu t}}(s),
\end{equation}
where $R,\check{R}^x$ are two independent BESQ($\mu$) starting from $1$ and $x\geq 0$ appropriately.
\end{theorem}
\begin{proof}
From Proposition \ref{plus} we have
\begin{equation*}
	R(t + s) = \check{R}^{R(t)}(s),
\end{equation*}
where $\check{R}^x$ is BESQ($\mu$) starting from $x\geq 0$ and independent of $R$. To finish the proof it is enough to put $A_t$ instead of $t$ in last equality.
\end{proof} 
\begin{remark} Observe that for $s = 0$ the last theorem is exactly the classical version of Lamperti's relation. In fact puting $A_s$ in place of $s$ we can also have from last result
$$
		R(A_t + A_s) = \check{R}^{e^{2B_t + 2\mu t}}(A_s).
$$

\end{remark}

\section{Distribution and conditional distribution of the first hitting time of squared Bessel process}
In first part of the section we present how the theory of squared Bessel time process introduced in previous sections enables us to describe the conditional distribution of the first hitting time of squared Bessel process. Again we assume that index $\mu$ is nonnegative and starting point of considered BESQ is equal to $1$. As above we denote by $\tau_y$ a first hitting time of $y>0$ by a BESQ $R$. Let $g_y$ be a density of the random variable $\tau_y$
$$
	g_y(t)dt = \mathbb{P}(\tau_y \in dt).
$$
Let $q_t$ denote the transition density of the squared Bessel process $R^x$ with index $\mu\geq 0$ and starting from $x$ 
$$
	q_t(x,y)dy = \mathbb{P}(R^x(t)\in dy).
$$
We have the following important theorem describing the conditional distribution $\mathbb{P}(\tau_y \leq t|R_T = x)$
\begin{theorem}\label{distr_cond} Let $t\leq T$ and $y > 0$. We have for $x>0$
\begin{equation}
	\mathbb{P}(\tau_y \leq t|R(T) = x) = \frac{1}{q_T(1,x)}\int_0^tq_{T-z}(y,x)g_y(z)dz.
\end{equation}
\end{theorem}
\begin{proof} From Theorem \ref{BesaM} we have the following representation 
$$
\frac{1}{y}R(\tau_y + a) = \check{R}_{\frac{a}{y}}, \ a\geq 0,
$$
where $\check{R}$ is independent of $\mathcal{F}^R_{\tau_y}$. Put $a = T - \tau_y$. Observe that on the set $\{\tau_y \leq t\}$ we have for positive Borel-measureable $f$ 
\begin{align*}
	1_{\{\tau_y\leq t\}}f(R(T)) = 	1_{\{\tau_y\leq t\}}f(\check{R}^y_{T-\tau_y}),
\end{align*}
where $\check{R}^y$ is starting from $y$ BESQ$(\mu)$, independent of $\mathcal{F}^R_{\tau_y}$. We have in result
\begin{align*}
	\mathbb{E}1_{\{\tau_y\leq t\}}f(R(T)) &= \mathbb{E}\Big(f(R(T))\mathbb{P}(\tau_y\leq t|R(T))\Big)\\
	&= \int_0^{\infty}f(x)\mathbb{P}(\tau_y\leq t|R(T) = x)q_T(1,x)dx.
\end{align*}
On the other hand independency between $\check{R}^y$ and $\mathcal{F}^R_{\tau_y}$ along with Fubinni's theorem yields
\begin{align*}
	\mathbb{E}1_{\{\tau_y\leq t\}}f(\check{R}^y_{T-\tau_y}) &= \mathbb{E}\int_0^{\infty}f(x)1_{\{\tau_y\leq t\}}q_{T-\tau_y}(y,x)dx\\
	&= \int_0^t\int_0^{\infty}f(x)q_{T-z}(y,x)g_y(z)dxdz\\
	&= \int_0^{\infty}f(x)\Big(\int_0^tq_{T-z}(y,x)g_y(z)dz\Big)dx.
\end{align*}
Thus from both equalities
\begin{align*}
	\mathbb{P}(\tau_y\leq t|R(T) = x)q_T(1,x) = \int_0^tq_{T-z}(y,x)g_y(z)dz.
\end{align*}
\end{proof}

The previous theorem states that conditional distribution of first hitting time of squared Bessel process can be expressed in terms of transition density function $q$ of BESQ$(\mu)$ and $g_y$ - the density of first hitting time $\tau_y$. As it was mentioned in Introduction the distribution of first hitting time of Bessel proces has been studied by several authors. The Laplace transform of the first hitting time of a point $y > 1$ can be computed as in \cite{K} by general theory on the eigenvalue expansion and is expressed as a ratio of modified Bessel functions (see Section 4 in Part II in \cite{SB}). The thorough study of first hitting time for linear continuous Markov processes can be found also in Chapter VII Section 3 in \cite{RY}. In case $y < 1$ the distribution of first hitting time of Bessel process can be found in recent works of Hamana and Matsumoto (\cite{HM1}, \cite{HM2}). When the index of Bessel process is $\frac12$ the density is expressed as simple formula, in other cases the density formula is expressed via series and integral representation using the zeros of the Bessel functions (see Theorem 2 in \cite{HM1}). 

In this paper we present completely new method of obtaining the distribution of first hitting time. Instead of considering a Bessel process we study the hitting time of squared Bessel process.
One can easily check that for nonnegative index the first hitting time of BESQ $\tau_y$ is the corresponding first hitting time of $\sqrt{y}$ of Bessel process. We use Theorem \ref{refprinc} to obtain the density of $\tau_y$ as a solution of an inhomogenuous Volterra integral equation of the first kind.

\begin{theorem}\label{distr} Let $y>0$. The density $g_y$ of the first hitting time $\tau_y$ of squared Bessel process with index $\mu \geq 0$ is a solution of the following integral equation
\begin{equation}
	q_t(1,x) = \frac{1}{y}\int_0^tq_{\frac{t-z}{y}}\Big(1,\frac {x}{y}\Big)g_y(z)dz, \ x \geq y, \ t >0,
\end{equation}
where $q$ is the transition density of squared Bessel process. 
Moreover, this solution is given by
$$
	g_y(t)  = \frac{\partial q_t(1,x)}{\partial t} + \int_0^tK(t-z) \frac{\partial q_z(1,x)}{\partial z}dz, \ t>0,
$$
where for $\lambda > 0$
\begin{align*}
	K(t) = \mathcal{L}^{-1}\Big( \frac{1}{\lambda G(\lambda)} \Big), \ G(\lambda) = \int_0^{\infty}e^{-\lambda s}q_{\frac{s}{y}}(1,x/y)ds
\end{align*}
and $\mathcal{L}^{-1}$ is the operator of inverse Laplace transform.
\end{theorem}
\begin{proof} Let $\hat{R}$ be BESQ($\mu$) independent of $R$. From Theorem \ref{refprinc} we have for $x\geq y$
\begin{align*}
	\mathbb{P}(R(t) < x) = 1 - \mathbb{P}\Big(\hat{R}\Big(\frac{t-\tau_y}{y}\Big)\geq \frac{x}{y}, \tau_y \leq t\Big).
\end{align*}
Thus from independence between $\hat{R}$ and $R$ and the last equality we obtain
\begin{align*}
	q_t(1,x) &= \mathbb{P}(R(t) \in dx)/ dx = -\frac{\partial}{\partial x} \mathbb{P}\Big(\hat{R}\Big(\frac{t-\tau_y}{y}\Big)\geq \frac{x}{y}, \tau_y \leq t\Big)\\
	&= -\frac{\partial}{\partial x}\int_{\frac{x}{y}}^{\infty}\int_0^t q_{\frac{t-z}{y}}(1,u)g_y(z)dzdu\\
	&= \frac{\partial}{\partial x}\int_{0}^{\frac{x}{y}}\int_0^t q_{\frac{t-z}{y}}(1,u)g_y(z)dzdu\\
	&= \frac{\partial}{\partial x}\int_{0}^{x}\int_0^t q_{\frac{t-z}{y}}(1,w/y)g_y(z)\frac{1}{y}dzdw\\
	&= \int_0^t q_{\frac{t-z}{y}}(1,x/y)g_y(z)\frac{1}{y}dz.
\end{align*}
The unique solution of the above inhomogenuous Volterra integral equation of the first kind and hence the density of $\tau_y$ is obtained from \cite{PM}
(or from Eqworld  \cite[Section 1.7 point 36]{E}):
$$
	g_y(t)  = \frac{\partial q_t(1,x)}{\partial t} + \int_0^tK(t-z) \frac{\partial q_z(1,x)}{\partial z}dz, \ t>0,
$$
where for $\lambda > 0$
\begin{align*}
	K(t) = \mathcal{L}^{-1}\Big( \frac{1}{\lambda G(\lambda)} \Big), \ G(\lambda) = \int_0^{\infty}e^{-\lambda s}q_{\frac{s}{y}}(1,x/y)ds
\end{align*}
and $\mathcal{L}^{-1}$ is the operator of inverse Laplace transform.
\end{proof}
 
From two above Theorems \ref{distr_cond} and \ref{distr} we can deduce another interesting property of trajectory of squared Bessel process: if $R(0) = 1$, $t>0$ and $R(t) > 1$ there is almost sure a point on interval $(0,t]$ such that $R$ reaches $1$.
\begin{corollary}Let $t> 0$ and $y=1$. If $R(0) =1$ then 
$
	\mathbb{P}(\tau_1\leq t|R(t) = x) = 1
$
for $x\geq 1$.
\end{corollary}
\begin{proof}If $y=1$ we have from Theorem \ref{distr_cond} for $T=t$
\begin{align*}
\mathbb{P}(\tau_1 \leq t|R(t) = x) = \frac{1}{q_t(1,x)}\int_0^tq_{t-z}(1,x)g_1(z)dz.
\end{align*}
However for $x\geq 1$ we have from Theorem \ref{distr}
\begin{align*}
q_t(1,x) = \int_0^tq_{t-z}(1,x)g_1(z)dz,
\end{align*}
hence $
	\mathbb{P}(\tau_1\leq t|R(t) = x) = 1.
$
\end{proof}

\bibliographystyle{plain}

\begin{thebibliography}{10}

\bibitem{ADY} Alili L., Dufresne D., Yor M.  \textit{Sur l'identit\'e de Bougerol pour les fonctionnelles expo
nentielles du mouvement brownien avec drift}, in \cite{YY} (1997), 3-14.

\bibitem{SB} Borodin A., Salminen P. \textit{Handbook of Brownian Motion - Facts and Formulae}. Birkhauser (2nd
ed.), 2002.

\bibitem{BMR} Byczkowski T., Malecki J., Ryznar M. \emph{Hitting times of Bessel processes} Volume 38, Issue 3 (2013), pp 753-786, 

\bibitem{DMY} Donati-Martin C., Matsumoto H., Yor M. \emph{Some absolute continuity relationship for certain anticipative transformations of geometric Brownian motion}, Publ. RIMS Kyoto Univ. 37 (2001), 295-326.

\bibitem{E} Eqworld - The World of Mathematical equations, http://eqworld.ipmnet.ru/index.htm

\bibitem{HM1} Hamana Y., Matsumoto H. \emph{The probabilty densities of the first hitting times of Bessel processes}, J.Math-for-Industry 4 (2012), 91-95.

\bibitem{HM2} Hamana Y., Matsumoto H. \emph{The probabilty distributions of the first hitting times of Bessel processes}, Trans. AMS 365 (2013), 5237-5257.

\bibitem{JYC} Jeanblanc M., Yor M., Chesney M. \textit{Mathematical Methods for Financial Markets}. Springer-Verlag London 2009.

\bibitem{K} Kent J. \emph{Eigenvalue expansions for diffusion hitting times}, Z. Wahr. Ver. Gebiete 52 (1980), 309-319.

\bibitem{LI} Lamperti J. \emph{Continuous state branching processes}, Bull. A.M.S. 73 (1967), 382-386.

\bibitem{LII} Lamperti J. \emph{Semi-stable Markov processes}, I. Z. W. 22 (1972), 205-255.
\bibitem{MYI} Matsumoto H., Yor M. \emph{A Relationship between Brownian motions with opposite drifts via certain enlargements of the Brownian filtration}, Osaka J. Math.  38 (2001), 383-398.

\bibitem{MYII} Matsumoto H., Yor M.  \emph{An analogue of Pitman's 2M - X theorem for exponential
Wiener functionals, part I: A time inversion approach}, Nagoya Math.
J. 159 (2000), 125-166.

\bibitem{MYIV} Matsumoto H., Yor M. \emph{Exponential
functionals of Brownian motion, I, Probability laws at fixed time},
Probab. Surveys 2 (2005), 312-347.

\bibitem{PM} Polyanin, A. D. and Manzhirov, A. V. \emph{Handbook of Integral Equations}, CRC Press, Boca Raton, 1998.

\bibitem{RY} Revuz D.,  Yor M. \emph{Continous Martingales and Brownian Motion}, Springer-Verlag (3rd
ed.), 2005.

\bibitem{RW}  Rogers C. G., Williams D. \textit{Diffusions, Markov Processes and Martingales: Volume 2, Itô Calculus}, Cambridge University Press (2nd ed.) (2000)

\bibitem{YY} Yor M. (Ed.) (1997) \textit{Exponential Functionals and Principal Values related to Brownian Motion.} A
collection of research papers, Biblioteca de la Revista Matem\'atica Iberoamericana.

\bibitem{Y} Yor M. \textit{On some exponential functionals of Brownian motion} Adv. App. Prob. 24 (1992), 509-531.

\end{thebibliography}

\end{document}